\documentclass[11pt]{amsart}
\usepackage[headings]{fullpage}
\usepackage{amssymb,epic,eepic,epsfig,amsbsy,amsmath,amscd,graphicx}
\usepackage[all]{xy}

\numberwithin{equation}{section}
                        \theoremstyle{plain}
\usepackage{mathrsfs}

\newtheorem{theorem}{Theorem}[section]

\newtheorem{lemma}[theorem]{Lemma}
\newtheorem{corollary}[theorem]{Corollary}
\newtheorem{proposition}[theorem]{Proposition}
\newtheorem{conjecture}[theorem]{Conjecture}

\theoremstyle{definition}
\newtheorem{remark}[theorem]{Remark}
\newtheorem{example}[theorem]{Example}

\newcommand{\ZZ}{{\mathbb{Z}}}

\newcommand{\QQ}{{\mathbb{Q}}}

\newcommand{\bdy}{{\partial}}

\newcommand{\K}{{K_{p, q}}}

\newcommand\no[1]{}

\newtheorem*{namedtheorem}{\theoremname}
\newcommand{\theoremname}{testing}

\def\BC{\mathbb C}
\def\BN{\mathbb N}
\def\BZ{\mathbb Z}

\def\CS{\mathcal S}

\def\ve{\varepsilon}
\def\be { \begin{equation} }
\def\ee { \end{equation} }

\begin{document}

\title[]{Knot Cabling and the Degree of the Colored Jones Polynomial II}

\author[]{Efstratia Kalfagianni}
\address{Department of Mathematics, Michigan State University, East Lansing, MI, 48824}
\email{kalfagia@math.msu.edu}

\author[]{Anh T. Tran}
\address{Department of Mathematical Sciences, University of Texas at Dallas,  Richardson TX, 75080}
\email{att140830@utdallas.edu}

\begin{abstract} We continue our study of the degree of the colored Jones polynomial 
under knot cabling started in \cite{Effie-Anh-slope}. Under certain hypothesis on this degree, we determine how the Jones slopes and the linear term behave under cabling. As an application we verify Garoufalidis' Slope Conjecture and a conjecture of 
\cite{Effie-Anh-slope} for cables of a two-parameter family of closed 3-braids called  2-fusion knots.

\bigskip

\bigskip

\bigskip

\noindent {2010 {\em Mathematics Classification:} {\rm Primary 57N10. Secondary 57M25.}\\

}
\end{abstract}
\thanks {\today}
\thanks{E. K. was partially supported in part by NSF grants DMS--1105843 and DMS--1404754 }

\maketitle

\section{Introduction}

For a knot $K \subset S^3$, let $n(K)$ denote a tubular neighborhood of
$K$ and let $M_K:=\overline{ S^3\setminus n(K)}$ denote the exterior of
$K$. Let $\langle \mu, \lambda \rangle$ be the canonical
meridian--longitude basis of $H_1 (\bdy n(K))$.  An element $a/b \in
{\QQ}\cup \{ 1/0\}$ is called a \emph{boundary slope} of $K$ if there
is a properly embedded essential surface $(S, \bdy S) \subset (M_K,
\bdy n(K))$, such that  $\bdy S$ represents $a \mu + b \lambda \in
H_1 (\bdy n(K))$.  Hatcher showed that every knot $K \subset S^3$
has finitely many boundary slopes \cite{hatcher}.
We will use $bs_K$ to denote the set of boundary slopes of $K$.

For a Laurent polynomial $f(v) \in \BC[v^{\pm 1/4}]$, let $d_+[f]$  be the highest degree of $f$ in $v$. For a knot $K \subset S^3$, Garoufalidis \cite{Ga-quasi} proved that the highest degree of its colored Jones polynomial is a quadratic quasi-polynomial. This means that there exist an integer $N_K>0$ and periodic functions $a_{K}(n), b_{K}(n), c_{K}(n)$ such that 
$$d_+[J_K(n)] = a_{K}(n) \, n^2 + b_{K}(n) n + c_{K}(n)$$
for $n \ge N_K$.  The common period of $a_{K}(n), b_{K}(n) ,  c_{K}(n)$ is called the period of $K$.
Let $d_-[J_K(n)]$ denote the lowest degree of $J_K(n)$ in $v$.
For a sequence $\{x_n\}$, let $\{x_n\}'$ denote the set of its cluster points. Elements of the set
$$js_K:= \left\{ 4n^{-2}d_+[J_K(n)]  \right\}' \quad
 \mbox{and} \quad js^*_K:= \left\{ 4n^{-2}d_-[J_K(n)] \right\}' $$
 are called {\em Jones slopes} of $K$. 
It is known that
 $J_K(n, v) = J_{K^*}(n, v^{-1})$, where $K^*$ is the mirror image of $K$. Thus statements about the lowest degree 
$d_-[J_K(n)]$, which is equal to $- d_+[J_{K^*}(n)]$,  can be equivalently made. With this in mind we will mostly discuss the highest degree in this paper.

In \cite {Effie-Anh-slope}  we studied the behavior of $d_+[J_K(n)]$ under knot cabling for knots of period at most two.
In this paper we study knots with period greater than two.

To state our results, let $K$ be a knot with framing 0, and $p,q$ are co-prime integers.
The $(p,q)$-cable $K_{p,q}$ of $K$ is the $0$-framed satellite of $K$ with
pattern $(p,q)$ torus knot. It is known that $K_{-p,-q}=\text{r}K_{p,q}$, where $\text{r}K_{p,q}$ denotes $\K$ with the opposite orientation, and that the colored Jones polynomial of a \textit{knot} is independent of the orientation of the knot. Hence in the statements of results below, and throughout the paper,  we will assume that our cables are non-trivial in the  sense  that $q>1$.

 \begin{theorem}
 \label{thm-quasi-constant}
 Let $K$ be a knot such that for $n \gg 0$ we have
$$
d_+[J_K(n)]=an^2+b(n)n+d(n)
$$
where $a$ is a constant, $b(n)$ and $d(n)$ are periodic functions with $b(n) \le 0$. Let \begin{eqnarray*}
M_1 &=&  \max\{|b(i)-b(j)| \, : \,  i \equiv j \pmod{2}\},\\
M_2 &=& \max\{2b(i) + |b(i)-b(j)| + |d(i)-d(j)| \, : \, i \equiv j \pmod{2}\}.
\end{eqnarray*}

Suppose $p- (4a-M_1)q < 0$ or $p- (4a + M_1)q > \max\{0,M_2\}$. Then for $n \gg 0$ we have
$$d_+[J_{\K}(n)]=An^2+B(n)n+D(n)$$
where $A$ is a constant, and $B(n), D(n)$ are periodic functions with $B(n) \le 0$. Moreover, if $js_K \subset bs_K$ then $js_{\K} \subset bs_{\K}$.
 \end{theorem}

Theorem \ref{thm-quasi-constant} has applications to two open conjectures about the degree of  $J_K(n)$.
The first conjecture, formulated by the authors in  \cite{Effie-Anh-slope},  asserts  that the linear part
$b_K(n)$ is not positive and it detects the existence of essential annuli in the knot complement.

\begin{conjecture} \cite{Effie-Anh-slope}
\label{conj}
For every non-trivial knot $K \subset S^3$, we have $$b_K(n) \le 0.$$
Moreover, if $b_K(n)=0$ then   $K$ is a composite knot or  a cable knot or a torus knot. 
\end{conjecture}

Note that $b_U(n)=1/2$ for the trivial knot $U$. 
 It is known that a knot $K$ is composite or cable or a torus knot if and only if its complement $M_K$ contains  embedded essential annuli \cite[Lemma V.1.3.4]{jacoshalen}. Thus the last part of Conjecture \ref{conj} can alternatively be stated as
follows: If $b_K(n)=0$, then $M_K$ contains an embedded essential  annulus.

The second conjecture is the Slope Conjecture formulated by Garoufalidis in  \cite{Ga-slope}.

\begin{conjecture} [Slope Conjecture] \cite{Ga-slope}
\label{slope}
For every knot $K \subset S^3$ we have 
$$js_K : = \{4a_K(n)\} \subset bs_K.$$
\end{conjecture}

Conjecture  \ref{slope} is known for the following classes of knots:
\begin{enumerate}

\item alternating knots, non-alternating knots with up to nine crossings, torus knots, the family ($-2, 3, 2m+3)$ of 3-string pretzel  knots (by Garoufalidis \cite{Ga-slope}).

\item adequate knots (by Futer, Kalfagianni and Purcell \cite{FKP}).

\item iterated cable knots of $B$-adequate knots,  iterated torus knots, and iterated of all non-alternating knots with up to nine crossings that have period two  (by Kalfagianni and Tran \cite{Effie-Anh-slope}).

\item a certain 2-parameter family of closed 3-braids, called 2-fusion knots (by Garoufalidis and Dunfield \cite{DG} and Garoufalidis and van der Veen \cite{GV}).

\item some families of 3-string pretzel knots (by Lee and van der Veen \cite{leV}).
\end{enumerate}

In this paper we will verify Conjecture \ref{conj} for 2-fusion knots and ``most" of their cables. For cases (1)-(3) 
the conjecture was verified in  \cite{Effie-Anh-slope}.
Theorem \ref{thm-quasi-constant} implies the following.

\begin{corollary} \label{cables}
 Let $K$ be a knot such that for $n \gg 0$ we have
$$
d_+[J_K(n)]=an^2+b(n)n+d(n)
$$
where $a$ is a constant, $b(n)$ and $d(n)$ are periodic functions. Let $M_1, M_2$
be as in the statement of Theorem \ref{thm-quasi-constant}.
If $K$ satisfies Conjectures \ref{conj}  and  \ref{slope}, then
$\K$ satisfies the conjectures  provided that  $p- (4a-M_1)q < 0$ or $p- (4a + M_1)q > \max\{0,M_2\}$. 
 \end{corollary}

As an application of  Corollary \ref{cables} we will prove Conjectures \ref{conj}  and  \ref{slope}
for cables of 2-fusion knots.

We would like to thank Stavros Garoufalidis for helpful correspondence.

\section{Jones slopes of cable knots}

Let $K_{p,q}$ be the $(p,q)$-cable of a knot $K$, where $q > 1$. For $n>0$, let $\CS_n$ be the set of all $k$ such that 
 $$
 |k| \le (n-1)/2 \quad \text{and} \quad k \in \begin{cases} \BZ &\mbox{if } n \text{ is odd}, \\ 
 \BZ+\frac{1}{2} & \mbox{if } n \text{ is even}. \end{cases}
 $$
By \cite{Ve}, for $n>0$ we have
\begin{equation}
\label{cables11}
J_{K_{p,q}}(n)= v^{pq(n^2-1)/4} \sum_{k \in \CS_n} v^{-pk(qk+1)} J_K(2qk+1),
\end{equation} 
where it is understood that $J_K(-m)=-J_K(m)$.

\begin{proposition}
\label{quasi-constant}
Let $K$ be a knot such that for $n \gg 0$ we have
$$
d_+[J_K(n)]=an^2+b(n)n+d(n)
$$
where $a$ is a constant, $b(n)$ and $d(n)$ are periodic functions with $b(n) \le 0$. Let
\begin{eqnarray*}
M_1 &=&  \max\{|b(i)-b(j)| : i  \equiv j \pmod{2} \},\\
M_2 &=& \max\{2b(i) + |b(i)-b(j)| + |d(i)-d(j)| : i \equiv j \pmod{2}\}.
\end{eqnarray*}

Suppose $p- (4a-M_1)q < 0$ or $p- (4a + M_1)q > \max\{0,M_2\}$. Then for $n \gg 0$ we have
$$d_+[J_{\K}(n)]=An^2+B(n)n+D(n)$$
where $A$ is a constant with $A \in \{q^2a, pq/4\}$, and $B(n), D(n)$ are periodic functions with $B(n) \le 0$.
\end{proposition}

\begin{proof} Fix $n \gg 0$. Recall the cabling formula \eqref{cables11} of the colored Jones polynomial
$$J_{K_{p,q}}(n)= v^{pq(n^2-1)/4} \sum_{k \in \CS_n} v^{-pk(qk+1)} J_K(2qk+1).$$
In the above formula, there is a sum. Under the assumption of the proposition, we will show that there is a unique term of the sum whose highest degree is strictly greater than those of the other terms. This implies that the highest degree of the sum is exactly equal to the highest degree of that unique term. 

For $k \in \CS_n$ let $$f(k) := d_+ [v^{-pk(qk+1)} J_K(2qk+1)] = -pk(qk+1)+d_+[J_K(|2qk+1|)].$$
The goal is to show that $f(k)$ attains its maximum on $\CS_n$ at a unique $k$. 

Since $d_+[J_K(n)]$ is a quadratic quasi-polynomial, $f(k)$ is a piece-wise quadratic polynomial. The above goal will be achieved in 2 steps. In the first step we show that $f(k)$ attains its maximum on each piece at a unique $k$. Then in the second step we show that the maximums of $f(k)$ on all the pieces are distinct.

\textbf{Step 1.} Let $\pi$ be the period of $d_+[J_K(n)]$. For $\ve \in \{\pm 1\}$ and $0 \le i < \pi$, let $h_i^{\ve}(x)$ be the quadratic real polynomial defined by
$$
h_i^{\ve}(x) := ( -pq + 4q^2 a ) x^2 + ( -p + 4q a + 2q b(i) \ve ) x +  a + b(i) \ve + d(i).
$$
For each $k \in \CS_n$, we have $f(k)= h^{\ve_k}_{i_k}(k)$
for a unique pair $(\ve_k, i_{k})$. Let
$$I_n := \{ (\ve_k,i_{k}) \mid k \in \CS_n\}.$$
Then $f(k)$ is a piece-wise quadratic polynomial of exactly $|I_n|$ pieces, each of which is associated with a unique pair $(\ve, i)$ in $I_n$.

For each $(\ve,i) \in I_n$, let $$\CS_{n,\ve,i} := \{ k \in \CS_n \mid (\ve_k, i_k) = (\ve,i)\}$$
which is the set of all $k$ on the piece associated with $(\ve,i)$. 

The quadratic polynomial $h_i^{\ve}(x)$ is concave up if $p-4qa<0$, and concave down if $p-4qa>0$.  Hence, for $n \gg 0$, $h_i^{\ve}(k)$ is maximized on the set $\CS_{n,\ve,i}$ at a unique $k=k_{n,\ve,i}$, where
$$ k_{n,\ve,i} := \begin{cases} \max \CS_{n,\ve,i} &\mbox{if } (p - 4qa)\ve < 0, \\
\min \CS_{n,\ve,i} &\mbox{if } (p - 4qa)\ve > 0.  \end{cases}$$
Note that, as in the proof of \cite[Proposition 3.2]{Effie-Anh-slope}, we use the assumption that $b(i) \le 0$ when $(p - 4qa)\ve > 0$. Moreover we have 
$$ \begin{cases}  |k_{n,\ve,i}|  \to \infty \mbox{ as }  n \to \infty, &\mbox{if } p - 4qa < 0 \\
 |k_{n,\ve,i}|  \le \pi, &\mbox{if } p - 4qa > 0.  \end{cases}$$

\textbf{Step 2.} Let 
$$\text{Max}_n := \max \{f(k) \mid k \in \CS_n\}.$$
From step 1 we have $\text{Max}_n =  \max \{ h_{i}^{\ve}(k_{n, \ve, i}) \mid (\ve, i) \in I_n \}.$ We claim that $$h_{i_1}^{\ve_1}(k_{n,\ve_1,i_1}) \not= h_{i_2}^{\ve_2}(k_{n,\ve_2,i_2})$$
for $(\ve_1,i_1) \not= (\ve_2,i_2)$. 

Indeed, let $k_1:=k_{n,\ve_1,i_1}$ and $k_2:=k_{n,\ve_2,i_2}$. Note that $k_1 \not= k_2$. Moreover, $k_1$ and $k_2$ are both in  $\BZ$ or $\frac{1}{2} + \BZ$. As a result we have $k_1 \pm  k_2 \in \BZ$, and $i_1 - i_2 \equiv 2q (k_1 - k_2) \equiv 0 \pmod{2}$. Let
$$\sigma := h_{i_1}^{\ve_1}(k_1) - h_{i_2}^{\ve_2}(k_2).$$

Without loss of generality, we can assume that $|k_1| \ge |k_2|$. Then we write $\sigma = \sigma' + d(i_1) - d(i_2)$ where
$$
\sigma' := \begin{cases} (k_1 - k_2) \Big( (-p+4qa) \big( q(k_1+k_2)+1 \big) + 2qb(i_1)\ve_1 \Big) \\ \qquad \qquad \qquad + \, \big( b(i_1) - b(i_2) \big) \ve_1 (2qk_2 + 1) &\mbox{if } \ve_1 = \ve_2, \\
\Big( (-p+4qa)(k_1 - k_2) + 2b(i_1)\ve_1 \Big) \big( q(k_1+k_2)+1 \big) 
\\ \qquad \qquad \qquad - \, \big( b(i_1) - b(i_2) \big) \ve_1 (2qk_2 + 1) &\mbox{if } \ve_1 \not= \ve_2.  \end{cases}
$$
We consider the following 2 cases.

\smallskip

\textit{Case 1:} $p - (4a - M_1)q <0 $. In particular, we have $p - 4qa <0$. There are 2 subcases.

\underline{Subcase 1.1:} $\ve_1 = \ve_2$. Since $k_1$ and $k_2$ have the same sign, we have 
$$|q(k_1+k_2)+1| - |2qk_2+1| = 2q (|k_1| - |k_2|) \ge 0.$$
Hence
\begin{eqnarray*}
|\sigma'| &\ge& \big| (-p+4qa) \big( q(k_1+k_2)+1 \big) + 2qb(i_1)\ve_1 \big| - \big| \big( b(i_1) - b(i_2) \big) \big( q(k_1+k_2)+1 \big) \big| \\
&\ge& \big( -p + 4qa - |b(i_1) - b(i_2)| \big) |q(k_1+k_2)+1| + 2q b(i_1).
\end{eqnarray*}
Since $|q(k_1+k_2)+1| \to \infty$ as $n \to \infty$, and $$-p + 4qa - |b(i_1) - b(i_2)| \ge -p + 4qa - M_1 >0$$ 
we get $|\sigma'| \to \infty$ as $n \to \infty.$

\smallskip

\underline{Subcase 1.2:} $\ve_1 \not= \ve_2$. Since $k_1$ and $k_2$ have opposite signs, we have 
$$(q|k_1-k_2| + 1) - |2qk_2+1| \ge 2q (|k_1| - |k_2|)  \ge 0.$$
Hence
\begin{eqnarray*}
|\sigma'| &\ge& \big|  (-p+4qa)(k_1 - k_2) + 2b(i_1)\ve_1 \big| -  | b(i_1) - b(i_2)| \, ( q|k_1 - k_2| + 1) \\
&\ge& \big( -p + 4qa - q |b(i_1) - b(i_2)| \big) |k_1 - k_2| -  | b(i_1) - b(i_2)|  + 2 b(i_1).
\end{eqnarray*}
Since $|k_1 - k_2| \to \infty$ as $n \to \infty$, and $$-p + 4qa - q |b(i_1) - b(i_2)| \ge -p + 4q a  - q M_1 >0,$$ we get $|\sigma'| \to \infty$ as $n \to \infty.$

\medskip 

\textit{Case 2:} $p- (4a + M_1)q > \max\{0, M_2\}$. There are 2 subcases.

\underline{Subcase 2.1:} $\ve_1 = \ve_2$. Note that $q(k_1+k_2)+1$ and $\ve_1$ have the same sign. Moreover, both $-p+4qa$ and $2qb(i_1)$ are non-positive. As in subcase 1.1 we have 
\begin{eqnarray*}
|\sigma'| &\ge& \big| (-p+4qa) \big( q(k_1+k_2)+1 \big) + 2qb(i_1)\ve_1 \big| - \big| \big( b(i_1) - b(i_2) \big) \big( q(k_1+k_2)+1 \big) \big| \\
&=& \big( p - 4qa - |b(i_1) - b(i_2)| \big) \, |q(k_1+k_2)+1| - 2q b(i_1).
\end{eqnarray*}
Since $p - 4qa - |b(i_1) - b(i_2)|  \ge p - 4qa - M_1 > \max \{0, M_2\}$, we get 
$$|\sigma'| > M_2  -2 q b(i_1) \ge |d(i_1) - d(i_2)|.$$

\underline{Subcase 2.2:} $\ve_1 \not= \ve_2$. Note that $k_1 - k_2$ and $\ve_1$ have the same sign. Moreover, both $-p+4qa$ and $2qb(i_1)$ are non-positive. As in subcase 1.2 we have
\begin{eqnarray*}
|\sigma'| &\ge& \big|  (-p+4qa)(k_1 - k_2) + 2b(i_1)\ve_1 \big| -  | b(i_1) - b(i_2)| \, ( q|k_1 - k_2| + 1) \\
&=& \big( p - 4qa - q |b(i_1) - b(i_2)| \big) |k_1 - k_2| - | b(i_1) - b(i_2)|  - 2 b(i_1).
\end{eqnarray*}
Since $p - 4qa - q |b(i_1) - b(i_2)| \ge p - 4qa - q M_1 > \max\{0, M_2\}$, we get $$|\sigma'| > M_2 -  | b(i_1) - b(i_2)|  -2 b(i_1) \ge |d(i_1) - d(i_2)|.$$

In all cases, for $n \gg 0$ we have $|\sigma'| > |d(i_1)-d(i_2)|.$ Hence $$\sigma = \sigma' + d(i_1) - d(i_2) \not= 0.$$
We have proved that $f(k)$ attains its maximum on $\CS_n$ at a unique $k$. More precisely, there exists a unique $(\ve_n, i_n) \in I_n$ such that $h_{i_n}^{\ve}(k_{n, \ve_n, i_n}) = \text{Max}_n$. 

Equation \eqref{cables} then implies that  
\begin{eqnarray*}
d_+[J_{K_{p,q}}(n)] &=& pq(n^2-1)/4 + h^{\ve_n}_{i}(k_{n,\ve_n,i_n}).
\end{eqnarray*}

If $p - 4qa < 0$ then $k_{n,\ve_n,i_n} = \ve_n (n/2 + s_n)$, where $s_n$ is a periodic sequence and $s_n \le -1/2$. We have
\begin{eqnarray*}
d_+[J_{K_{p,q}}(n)] &=& q^2a n^2 + \big( (-p + 4qa) (q s_n + \ve_n /2) + qb(i_n) \big) n -  pq/4 \\
&& + \, (-p + 4qa)s_n (q s_n + \ve_n) + 2q b(i_n) s_n+  a + b(i_n) \ve_n + d(i_n).
\end{eqnarray*}
In this case we have $$B(n)=(-p + 4qa) (q s_n + \ve_n /2) + qb(i_n) < 0,$$ since $q s_n + \ve_n /2 \le -q/2 + 1/2 <0$ and $b(i_n) \le 0$.

If $p - 4qa > 0$ then $k_{n,\ve_n,i_n} = s_n$, where $s_n$ is a periodic function. We have
\begin{eqnarray*}
d_+[J_{K_{p,q}}(n)] &=& pq (n^2-1)/4 
+ (-p + 4qa)s_n (q s_n + 1) + 2q b(i_n) \ve_n s_n\\
&& + \,  a + b(i_n) \ve_n + d(i_n).
\end{eqnarray*}
In this case we have $B(n)=0$.

This completes the proof of Proposition \ref{quasi-constant}.
\end{proof}

\begin{remark} The proof of Proposition \ref{quasi-constant} can be slightly modified to give the following.

Let $K$ be a knot such that for $n \gg 0$ we have
$$
d_+[J_K(n)]=a(n)n^2+b(n)n+d(n)
$$
where $a(n)$, $b(n)$ and $d(n)$ are periodic functions with $b(n) \le 0$. Let
$$a_M = \max\{a(n)\}, \quad a_m = \min\{a(n)\} $$
and
\begin{eqnarray*}
M_1 &=&  \max\{|b(i)-b(j)| : i  \equiv j \pmod{2} \},\\
M_2 &=& \max\{2b(i) + |b(i) - b(j)| +|d(i)-d(j)| : i \equiv j \pmod{2}\}.
\end{eqnarray*}

(1) Suppose $p- (4 a_m -M_1)q < 0$. Then for $n \gg 0$ we have
$$d_+[J_{\K}(n)]=A(n)n^2+B(n)n+D(n)$$
where $A(n)$, $B(n)$ and $D(n)$ are periodic functions with $\{A(n)\} \subset \{q^2 a(n)\}$ and $B(n) < 0$.

(2) Suppose $\{a(n)\}$ has period at most $2$ and  $p- (4a_M + M_1)q > \max\{0,M_2\}$. Then for $n \gg 0$ we have
$$d_+[J_{\K}(n)]=pqn^2/4+O(1).$$
\end{remark}

We now recall the following result about the behavior of the boundary slopes under knot cabling in \cite{Effie-Anh-slope}.

\begin{theorem} \cite{Effie-Anh-slope}
\label{bdry slope}
 For every knot $K \subset S^3$ and  $(p, q)$ coprime  integers
we have 
 $$\left( q^2 bs_K \cup \{pq\} \right) \subset bs_{K_{p,q}}.$$
\end{theorem}

\no{
Here we study knots for which $d_+[J_K(n)]$ has a unique  Jones slope. We have provided conditions under which
if a number $a\in \QQ$ is a Jones slope of a mono-sloped knot $K$ then $q^2a$ is a Jones slopes of the cable knot $\K$.}

Proposition \ref{quasi-constant} and Theorem \ref{bdry slope} imply Theorem \ref{thm-quasi-constant} stated in the introduction.

\begin{example}  

\label{3-knots}

 Theorem \ref{thm-quasi-constant} applies to the non-alternating knots $8_{20}, 9_{43}, 9_{44}$ (of period $3$) in Section 4 of \cite{Ga-slope}. 
 
 For $K=8_{20}$ we have
 \begin{eqnarray*}
d_+[J_K(n)] &=& \begin{cases} 2n^2/3 - n/2 - 1/6 &\mbox{if } n \not\equiv 0 \pmod{3}\\ 
2n^2/3 - 5n/6 -1/2 & \mbox{if } n \equiv 0 \pmod{3}.\end{cases} 
\end{eqnarray*}
Hence $K_{p,q}$ satisfies Conjectures \ref{conj}  and  \ref{slope} if $p-\frac{7}{3}q<0$ or $p-3q > 0$.

For $K=9_{43}$ we have
 \begin{eqnarray*}
d_+[J_K(n)] &=& \begin{cases} 8n^2/3 - n/2 - 13/6 &\mbox{if } n \not\equiv 0 \pmod{3}\\ 
8n^2/3 - 5n/6 -7/2 & \mbox{if } n \equiv 0 \pmod{3}.\end{cases}
\end{eqnarray*}
Hence $K_{p,q}$ satisfies Conjectures \ref{conj}  and  \ref{slope} if $p-\frac{31}{3}q<0$ or $p - 11q>\frac{2}{3}$.

For $K=9_{44}$ we have
\begin{eqnarray*}
d_+[J_K(n)] &=& \begin{cases} 7n^2/6 - n - 1/6 &\mbox{if } n \not\equiv 0 \pmod{3}\\ 
7n^2/6 - 4n/3 -1/2 & \mbox{if } n \equiv 0 \pmod{3}.\end{cases}
\end{eqnarray*}
Hence $K_{p,q}$ satisfies Conjectures \ref{conj}  and  \ref{slope} if $p-\frac{13}{3}q<0$ or $p - 5q > 0 $.
\end{example}

\section{Two-fusion knots}

The family of 2-fusion knots is a two-parameter family  of closed 3-braids denoted
by
$$\{ K(m_1, m_2)  \mid  m_1, m_2\in \ZZ\}.$$
For the precise definition and description see \cite {GV, DG}. 

Note that $K(m_1, m_2)$ is a torus knot if $m_2 \in \{-1,0\}$. In fact, $K(m_1, 0) = T(2, 2m_1 + 1)$ and $K(m_1, -1) = T(2, 2m_1 -3)$. It is known that $K(m_1, m_2)$ is hyperbolic if $m_1 \notin \{0, 1\}$, $m_2 \notin \{-1,0\}$ and $(m_1, m_2) \not= (-1,1)$. See \cite{GV}. Note that $K(-1,1) = T(2,5)$.

From now on we consider $m_2 \notin \{-1,0\}$ only. Let 
\begin{eqnarray*}
I_1 &=& \{(m_1, m_2) : m_2 \ge 1 \text{ and } m_1 \ge 2\}, \smallskip \\
I_2 &=& \{(m_1, m_2) : m_2 \ge 1 \text{ and } m_1 = 1\}, \smallskip \\
I_3 &=& \{(m_1, m_2) : m_2 \ge 1 \text{ and } m_1  \le -(m_2 + 1)\} \\
&& \cup \, \{(m_1, m_2) : m_2 \ge 1 \text{ and } -(m_2 + 1)/2 > m_1 > -(m_2 + 1)\},\smallskip\\
I_4 &=& \{(m_1, m_2) : m_2 \ge 1 \text{ and } 0 \ge m_1 \ge - (m_2 + 1)/2\}, \smallskip\\
I_5 &=& \{(m_1, m_2) : m_2 \le -2 \text{ and }  m_1 \le -3m_2/2\}, \smallskip\\
I_6 &=& \{(m_1, m_2) : m_2 \le -2 \text{ and } m_1 > -3m_2/2\}.
\end{eqnarray*}
and let $\CS_{m_1, m_2}$ be the set of coprime pairs $(p,q)$, with $q>1$, such that

\begin{itemize}

\item  If $ (m_1, m_2) \in I_1$, then   $p- \left( 4m_1 + 8m_2 + 2 + \frac{m_2^2}{ m_1 + m_2 - 1} \right) q < 0$ or 

$p- \left( 4m_1 + 8m_2 + 2 + \frac{m_2^2}{ m_1 + m_2 - 1} \right) q > \frac{(1 - m_1 + m_2)^2}{4(m_1 + m_2 -1)}$.

\smallskip

\item  If $ (m_1, m_2) \in I_2$,  then $p- ( 9 m_2 + 6) q > \frac{3m_2}{4} \mbox{ or } p- ( 9 m_2 + 6) q < 0$.



\smallskip

\item If $ (m_1, m_2) \in I_4$, then $p-\left( 3m_1 + 9m_2 + 3 + \frac{m_1^2}{m_1 + m_2+1} \right) q < 0$ or 

$p- \left( 3m_1 + 9m_2 + 3 + \frac{m_1^2}{m_1 + m_2+1} \right) q > \max\{  \frac{m_1 + m_2 -1}{4} + \frac{m_1 (m_2 - 1)}{m_1 + m_2 + 1}, 0\}$.



\smallskip

\item If $(m_1, m_2) \in I_6$, then
$p- \frac {2(2 m_1 + 3 m_2)^2 + m_2 + 1} {2 m_1 + 2 m_2-1}  q < 0$ or

$p- \frac {2(2 m_1 + 3 m_2)^2 - (m_2 + 1)} {2 m_1 + 2 m_2-1}  q > \max\{ \frac{2m_1 + 2m_2 -1}{8} + \frac{(2m_1-6) (m_2+1)}{2m_1 + 2m_2 - 1}, 0\}$.
\end{itemize}

\begin{theorem}
\label{thm:slope-cable} Let $K=K(m_1, m_2)$ be a  2-fusion knot. Then, 

\begin{enumerate}
\item Conjectures  \ref{conj} and \ref{slope}  hold  for $K$.
\item Conjectures \ref{conj} and \ref{slope} hold true for $(p,q)$-cables of the 2-fusion knot $K(m_1, m_2)$, where $(p,q) \in \CS_{m_1,m_2}$.
\end{enumerate}
\end{theorem}

Due to different conventions and normalizations of the colored Jones polynomial, the degree $d_+[J_K(n)]$ in our paper is different from $\delta_K(n)$ in \cite{GV}. For $n>0$ we have $$d_+[J_K(n)] =  \delta_K(n-1) + (n-1)/2.$$


For $n \in \BN$ and $k_1, k_2 \in \BZ$ such that $0 \le k_1 \le n$ and $|n-2k_1| \le n+2k_2 \le n+2k_1$, let
\begin{eqnarray*}
Q(n,k_1,k_2)&=& \frac{k_1}{2} - \frac{3 k_1^2}{2} - 3 k_1 k_2 - k_2^2 - k_1 m_1 - k_1^2 m_1 - k_2 m_2 - k_2^2 m_2 - 6 k_1 n \\ 
& & - \, 3 k_2 n + 2 m_1 n + 4 m_2 n - k_2 m_2 n - 2 n^2 + m_1 n^2 + 2 m_2 n^2 \\ 
& & + \, \frac{1}{2} \left( (1 + 8 k_1 + 4 k_2 + 8 n) \min\{l_1,l_2,l_3\} - 3 \min\{l_1,l_2,l_3\}^2\right)
\end{eqnarray*}
where
$$
l_1=2 k_1 + n, \qquad l_2=2 k_1 + k_2 + n, \qquad l_3=k_2 + 2 n.
$$

\subsection{ The highest degree and the Jones slope.} The quantity  $Q(n,k_1,k_2)$ is closely related
to $\delta_K(n)$. According to \cite{GV}  for the 2-fusion knot $K=K(m_1,m_2)$, with $m_2 \notin \{-1, 0\}$, we have the following possibilities:
\vskip 0.04in

\textbf{Case A.} $m_1, m_2 \ge 1$. Then
$$\delta_K(n) = Q(n,k_1,-k_1),$$

\noindent where $$c_1=\frac{1 - m_1 + m_2 + m_2 n}{2 (-1 + m_1 + m_2)},$$ 
and  $k_1$ is of the integers closest to $c_1$ satisfying $k_1 \le n/2$. 

\vskip 0.04in

\textbf{Case B.} $m_1 \le 0, m_2 \ge 1$. There are 2 subcases.

\textbf{(B-1)} ($1 + m_1 + m_2 \le 0$) or ($1 + m_1 + m_2 > 0$ and $1 + 2m_1 + m_2 < 0$). Then
$$\delta_K(n) = Q(n,n,0).$$

\textbf{(B-2)} $1 + m_1 + m_2 > 0$ and $1 + 2m_1 + m_2 \ge 0$. 
 Then
$$\delta_K(n) = Q(n,k_1,k_1-n),$$
\noindent where  $$c_2=\frac{1 - m_1 - m_2 + (1 + m_2) n}{2 (1 + m_1 + m_2)}.$$
and $k_1$ is one of the integers closest to $c_2$.

\textbf{Case C.} $m_2 \le -2$. There are 2 subcases.

\textbf{(C-1)} $m_1 \le -3m_2/2$. Then 
$$\delta_K(n) = Q(n,n,n).$$

\textbf{(C-2)} $m_1 > -3m_2/2$. Let $$c_3=\frac{-3/2 + m_1 + m_2 + (1+m_2 ) n}{1-2m_1-2m_2}$$
and let $k_1$ be one of the integers closest to $c_3$. Then
$$\delta_K(n) = \begin{cases} Q(n,k_1,k_1) &\mbox{if } c_3 \notin \frac{1}{2} + \BZ \\ 
Q(n,k_1,k_1)-(c_3 + 1/2) & \mbox{if } c_3 \in \frac{1}{2} + \BZ \end{cases}$$

We need the following that gives the Jones slopes of 2-fusion knots.

\begin{theorem} \label{a_K}
\cite{DG, GV}
The slope conjecture holds for 2-fusion knot $K=K(m_1,m_2)$. Moreover, we have
$$
 a_K(n) = \begin{cases} m_1 + 2m_2 + \frac{1}{2} + \frac{m_2^2}{4(-1 + m_1 + m_2)} &\mbox{if } (m_1, m_2) \in I_1 \cup I_2, 
\smallskip
\\ 1/2 + 2 m_2   & \mbox{if } (m_1, m_2) \in I_3, 
\smallskip
\\ \frac{3 + 3m_1 + 9m_2}{4} +  \frac{m_1^2}{4(1 + m_1 + m_2)} & \mbox{if } (m_1, m_2) \in I_4, 
\smallskip
\\ 0  & \mbox{if } (m_1, m_2) \in I_5, 
\smallskip
\\ \frac {(2 m_1 + 3 m_2)^2} {2 (-1 + 2 m_1 + 2 m_2)} & \mbox{if } (m_1, m_2) \in I_6.
\end{cases}
$$
\end{theorem}

\subsection{Calculating the linear term}  

\begin{theorem}

\label{prop:b}

For the 2-fusion knot $K=K(m_1,m_2)$, with $m_2 \notin \{-1,0\}$, we have
$$
b_K(n) = \begin{cases} \frac{m_2(1-m_1)}{2(-1+m_1+m_2)} &\mbox{if } (m_1, m_2) \in I_1 \cup I_2, 
\smallskip
\\ 1 + m_1  & \mbox{if } (m_1, m_2) \in I_3, 
\smallskip
\\ \frac{m_1(m_2-1)}{2(1+m_1+m_2)}   & \mbox{if } (m_1, m_2) \in I_4, 
\medskip
\\ 5/2 + m_1 + 3 m_2  & \mbox{if } (m_1, m_2) \in I_5, 
\smallskip
\\ \frac {(-5 + 2m_1)(1 + m_2)} {2 (-1 + 2 m_1 + 2 m_2)} & \mbox{if } (m_1, m_2) \in I_6 \mbox{ and } \frac {-1+(1 + m_2)(n-1)} {-1 + 2 m_1 + 2 m_2} \notin  \BZ,
\medskip
\\ \frac {(-3 + 2m_1)(1 + m_2)} {2 (-1 + 2 m_1 + 2 m_2)} & \mbox{if } (m_1, m_2) \in I_6\mbox{ and } \frac {-1+(1 + m_2)(n-1)} {-1 + 2 m_1 + 2 m_2} \in  \BZ.
\end{cases}
$$
In particular we have $b_K(n) \le 0$. Moreover $b_K(n) = 0$ if and only if $m_1 \in \{0,1\}$ and $m_2 \ge 1$, or $(m_1, m_2) = (-1,1)$.
\end{theorem}

\begin{proof} As in the previous subsection, there are 3 cases.

\vskip 0.04in

\textbf{Case A.} $m_1, m_2 \ge 1$. Recall that $$c_1=\frac{1 - m_1 + m_2 + m_2 n}{2 (-1 + m_1 + m_2)}$$ 
and $k_1$ is one of the integers closest to $c_1$ satisfying $k_1 \le \frac{n}{2}$. We have
\begin{eqnarray*}
\delta_K(n) = Q(n,k_1,-k_1)
&=& (1-m_1-m_2) k_1^2 + ( 1 - m_1 + m_2 + m_2 n) k_1 \\
&& + \, 2 m_1 n + 4 m_2 n  + m_1 n^2 + 2 m_2 n^2 + \frac{n}{2} + \frac{n^2}{2}.
\end{eqnarray*}

Write $k_1 = c_1 + r_n$ where $r_n$ is a periodic sequence with $\begin{cases} |r_n| \le 1/2 &\mbox{if } m_1 \ge 2, 
\\ r_n \in \{-1/2, -1\}  & \mbox{if } m_1 = 1. 
\end{cases}$ We have
$$
\delta_K(n) = Q(n,c_1+r_n,-c_1-r_n) = Q(n,c_1,-c_1) + (1-m_1-m_2) r_n^2
$$
and
\begin{eqnarray*}
Q(n,c_1,-c_1) &=& 2 m_1 n + 4 m_2 n  + m_1 n^2 + 2 m_2 n^2 + \frac{n}{2} + \frac{n^2}{2} -\frac{( 1 - m_1 + m_2 + m_2 n)^2}{4(1-m_1-m_2)} \\
&=& \left( m_1 + 2m_2 + \frac{1}{2} + \frac{m_2^2}{4(-1 + m_1 + m_2)} \right) n^2 \\
&& + \,  \left( 2m_1 + 4m_2 + \frac{1}{2} + \frac{m_2(1-m_1+m_2)}{2(-1 + m_1 + m_2)} \right) n - \frac{( 1 - m_1 + m_2)^2}{4(1-m_1-m_2)}.
\end{eqnarray*} 

Since $d_+[J_K(n)] = \delta_K(n-1) + (n-1)/2$ we obtain
\begin{eqnarray*}
d_+[J_K(n)] &=& \left( m_1 + 2m_2 + \frac{1}{2} + \frac{m_2^2}{4(-1 + m_1 + m_2)} \right) n^2 + \frac{m_2(1-m_1)}{2(-1+m_1+m_2)} n\\
&& - \, \left( m_1 + 2m_2 + \frac{1}{2} - \frac{(1-m_1)^2}{4(-1 + m_1 + m_2)} \right) + (1-m_1-m_2) r^2_{n-1}.
\end{eqnarray*} 

\textbf{Case B.} $m_1 \le 0, m_2 \ge 1$.  There are 2 subcases.

\textbf{(B-1)} ($1 + m_1 + m_2 \le 0$) or ($1 + m_1 + m_2 > 0$ and $1 + 2m_1 + m_2 < 0$). Then
$$\delta_K(n) = Q(n,n,0)=\left( \frac{1}{2} + 2 m_2 \right) n^2 + \left( \frac{3}{2} + m_1 + 4 m_2 \right) n.$$
Hence $$
d_+[J_K(n)] = \left( \frac{1}{2} + 2 m_2 \right) n^2 + (1 + m_1) n - \left( \frac{3}{2} + m_1 + 2 m_2 \right).
$$

\textbf{(B-2)} $1 + m_1 + m_2 > 0$ and $1 + 2m_1 + m_2 \ge 0$. Recall that $$c_2=\frac{1 - m_1 - m_2 + (1 + m_2) n}{2 (1 + m_1 + m_2)}$$
and $k_1$ is one of the integers closest to $c_2$. We have 
\begin{eqnarray*}
\delta_K(n) = Q(n,k_1,k_1-n)
&=& (-1-m_1-m_2) k_1^2 + ( 1 - m_1 - m_2 + (1 + m_2) n) k_1 \\
&& + \, 2 m_1 n + 5 m_2 n  + m_1 n^2 + 2 m_2 n^2 + \frac{n}{2} + \frac{n^2}{2}.
\end{eqnarray*}
Write $k_1 = c_2 + r_n$ where $r_n$ is a periodic sequence with $|r_n| \le 1/2$.  As in Case A we have 
$$\delta_K(n) = Q(n,c_2,c_2-n) + (-1-m_1-m_2)r_n^2$$
and
\begin{eqnarray*}
Q(n,c_2,c_2-n) &=&  2 m_1 n + 5 m_2 n  + m_1 n^2 + 2 m_2 n^2 + \frac{n}{2} + \frac{n^2}{2} - \frac{( 1 - m_1 - m_2 + (1 + m_2) n)^2}{4(-1-m_1-m_2)} \\
&=& \left( \frac{3}{4} + \frac{3m_1}{4} + \frac{9m_2}{4} +  \frac{m_1^2}{4(1 + m_1 + m_2)} \right) n^2  \\
&& + \,  \left( 1 + 2m_1 + \frac{9m_2}{2}  - \frac{m_1}{1 + m_1 + m_2} \right) n - \frac{( 1 - m_1 - m_2)^2}{4(-1-m_1-m_2)}.
\end{eqnarray*} 
Hence 
\begin{eqnarray*}
d_+[J_K(n)] &=& \left( \frac{3}{4} + \frac{3m_1}{4} + \frac{9m_2}{4} +  \frac{m_1^2}{4(1 + m_1 + m_2)} \right) n^2 + \frac{m_1(m_2-1)}{2(1+m_1+m_2)} n\\
&& - \,  \left( \frac{3}{4} + \frac{3m_1}{4} + \frac{9m_2}{4} - \frac{(m_2 - 1)^2}{4(1 + m_1 + m_2)}\right) + (-1-m_1-m_2)r_n^2.
\end{eqnarray*}

\textbf{Case C.} $m_2 \le -2$. There are 2 subcases.

\textbf{(C-1)} $m_1 \le -3m_2/2$. Then 
$$\delta_K(n) = Q(n,n,n)=(2 + m_1 + 3 m_2) n.$$
Hence $$d_+[J_K(n)] = (5/2 + m_1 + 3 m_2) (n-1).$$ 

\textbf{(C-2)} $m_1 > -3m_2/2$. Recall that $$c_3=\frac{-3/2 + m_1 + m_2 + (1+m_2 ) n}{1-2m_1-2m_2}$$
and let $k_1$ be one of the integers closest to $c_3$. We have
$$\delta_K(n) = \begin{cases} Q(n,k_1,k_1) &\mbox{if } c_3 \notin \frac{1}{2} + \BZ \\ 
Q(n,k_1,k_1)-(c_3 + 1/2) & \mbox{if } c_3 \in \frac{1}{2} + \BZ \end{cases}$$
and
\begin{eqnarray*}
Q(n,k_1,k_1)
&=& \left( 1/2-m_1-m_2 \right) k_1^2 - \left( -3/2 + m_1 + m_2 + (1+m_2 ) n \right) k_1 \\
&& + \, 2 m_1 n + 4 m_2 n  + m_1 n^2 + 2 m_2 n^2 + \frac{n}{2} + \frac{n^2}{2}.
\end{eqnarray*}

Write $k_1 = c_3 + r_n$ where $r_n$ is a periodic sequence with $|r_n| \le 1/2$. As in Case A we have 
$$Q(n,k_1,k_1) = Q(n,c_3,c_3) + \left( 1/2-m_1-m_2 \right) r_n^2$$
and
\begin{eqnarray*}
Q(n,c_3,c_3) &=&  2 m_1 n + 4 m_2 n  + m_1 n^2 + 2 m_2 n^2 + \frac{n}{2} + \frac{n^2}{2} - \frac{\left( -3/2 + m_1 + m_2 + (1+m_2 ) n \right)^2}{4(1/2-m_1-m_2)} \\
&=& \frac {(2 m_1 + 3 m_2)^2} {2 (-1 + 2 m_1 + 2 m_2)} n^2 + \left( \frac{1}{2} + 2 m_1 + \frac{9m_2}{2} + \frac {-3 + 2 m_1} {2 (-1 + 2 m_1 + 2 m_2)} \right) n\\
&& - \, \frac{\left( -3/2 + m_1 + m_2  \right)^2}{4(1/2-m_1-m_2)} .
\end{eqnarray*} 
Hence 
\begin{eqnarray*}
d_+[J_K(n)] &=& \frac {(2 m_1 + 3 m_2)^2} {2 (-1 + 2 m_1 + 2 m_2)} n^2 + 
\begin{cases} \frac {(-5 + 2m_1)(1 + m_2)} {2 (-1 + 2 m_1 + 2 m_2)} n &\mbox{if } \frac {-1+(1 + m_2)(n-1)} {-1 + 2 m_1 + 2 m_2}  \notin \BZ \medskip \\ 
\frac {(-3 + 2 m_1)(1 + m_2)} {2(-1 + 2 m_1 + 2 m_2)}n & \mbox{if } \frac {-1+(1 + m_2)(n-1)} {-1 + 2 m_1 + 2 m_2}  \in  \BZ \end{cases} \\
&& - \, \left( \frac{1}{2} + m_1 + 2m_2 - \frac{\left( 2m_1 -5  \right)^2}{8(2m_1 + 2m_2 -1)} \right) + \left( 1/2-m_1-m_2 \right) r_{n-1}^2.
\end{eqnarray*}
This completes the proof of Theorem \ref{prop:b}.
\end{proof}

The proof of Theorem \ref{prop:b} implies the following.

\begin{lemma} \label{M}
For the 2-fusion knot $K=K(m_1, m_2)$ we have
$$
M_1 = \begin{cases}  0  & \mbox{if } (m_1, m_2) \notin I_6, 
\\ \frac {1 + m_2} {-1 + 2 m_1 + 2 m_2} & \mbox{if } (m_1, m_2) \in I_6,
\end{cases}
$$
and 
$$
\max\{0, M_2\} = \begin{cases} \frac{(1 - m_1 + m_2)^2}{4(m_1 + m_2 -1)} &\mbox{if } (m_1, m_2) \in I_1, 
\smallskip\\
3m_2/4 &\mbox{if } (m_1, m_2) \in I_2, 
\smallskip
\\ 0  & \mbox{if } (m_1, m_2) \in I_3, 
\smallskip
\\ \max \big\{  \frac{m_1 + m_2 -1}{4} + \frac{m_1 (m_2 - 1)}{m_1 + m_2 + 1}, 0 \big\}   & \mbox{if } (m_1, m_2) \in I_4, 
\smallskip
\\ 0  & \mbox{if } (m_1, m_2) \in I_5, 
\smallskip
\\ \max \big\{ \frac{2m_1 + 2m_2 -1}{8} + \frac{(2m_1-6) (m_2+1)}{2m_1 + 2m_2 - 1}, 0 \big\} & \mbox{if } (m_1, m_2) \in I_6.
\end{cases}
$$
\end{lemma}
 
\subsection{Proof of Theorem \ref{thm:slope-cable}} 
Theorem \ref{prop:b} implies that Conjecture \ref{conj} holds true for 2-fusion knots: The fact that $b_K(n)\leq 0$ is clear by the statement
of Theorem \ref{prop:b}.  This, together with Theorem \ref{a_K}, proves Theorem \ref{thm:slope-cable} part $(1)$.
Moreover, $b_K(n) = 0$ if and only if $m_1 \in \{0,1\}$ and $m_2 \ge 1$, or $(m_1, m_2) = (-1,1)$.
As noted in \cite{GV}  we have
$K(1, m_2)= K^*(0, -m_2-1)$. On the other hand, by definition 
the knot $K(0, m_2)$ is a torus  knot. Finally  $K(-1, 1)$ is the torus knot $T(2,5)$.
Thus if $b_K(n)=0$, and $K=K(m_1, m_2)$, then $K$ is a torus knot.
By combining Theorem \ref{thm-quasi-constant}, Theorem \ref{a_K} and Lemma \ref{M} we get Theorem \ref{thm:slope-cable} part $(2)$.\qed
 
\begin{example} 
Consider the 2-fusion knot $K(m, 1)$, also known as the $(-2,3,2m+3)$-pretzel knot. It is known that $K(m,1)$ is $B$--adequate if $m \le -2$ and is $A$--adequate if $m \ge -1$. Moreover $K(m,1)$ is a torus knot if $|m| \le 1$, and $K(-2,1)$ is the twist knot $5_2$ which is an adequate knot. Hence we consider the two cases $m \ge 2$ and $m \le -3$ only. 

Note that $K(m_1,m_2)$ is the mirror image of $K(1-m_1, -1-m_2)$. In particular, $K(m,1)$ is the mirror image of $K(1-m, -2)$.

\textit{Case 1:} $m \ge 2$. From the proof of Theorem \ref{prop:b} we have
$$
d_+[J_{K(m, 1)}(n)] = \left( \frac{5}{2} + m + \frac{1}{4m} \right) n^2 + \left( \frac{1}{2m} - \frac{1}{2} \right) n - \left( 3 + \frac{3m}{4} - \frac{1}{4m} \right) - m r^2_{n-1}
$$
where $r_n$ is a periodic sequence with $|r_n| \le 1/2$.

Hence, by Corollary \ref{cables}, the $(p,q)$-cable of $K(m,1)$ satisfies Conjectures \ref{conj} and \ref{slope} if $$p-\left( 10 + 4m + \frac{1}{m} \right)q <0 \quad \text{or} \quad p-\left( 10+ 4m + \frac{1}{m} \right)q > \frac{m}{4} + \frac{1}{m} - 1.$$

\textit{Case 2:} $m \le -3$. From the proof of Theorem \ref{prop:b} we have
$$d_+[J_{K(1-m, -2)}(n)] = -\frac {2(m+2)^2} {2m+3} \, n^2 + b(n) n + (6m+17)/8 + (m + 3/2) r^2_{n-1}$$
where $r_n$ is a periodic sequence with $|r_n| \le 1/2$, and $b(n) = \begin{cases} -\frac{1}{2}   &\mbox{if } (2m+3) \nmid n, \medskip \\ 
-\frac {2m+1} {2( 2m+3)} & \mbox{if } (2m+3) \mid n. \end{cases}$

\smallskip

Hence, by Corollary \ref{cables}, the $(p,q)$-cable of $K(1-m, -2) = (K(m,1))^*$ satisfies Conjectures \ref{conj} and \ref{slope} if $$p + \left( 10 + 4m + \frac{1}{2m + 3} \right)q < 0 \quad \text{or} \quad p + \left( 10+ 4m + \frac{3}{2m+3} \right)q > - \left( \frac{m}{4} + \frac{1}{2m+3}+ \frac{11}{8} \right).$$
\end{example}


\no{ Similarly, let $K$ be a knot such that for $n \gg 0$ we have
$$
d_-[J_K(n)]=a^*n^2+b^*(n)n+d^*(n)
$$
where $a^*$ is a constant, $b^*(n)$ and $d^*(n)$ are periodic functions, and $b^*(n) \ge 0$. Let
\begin{eqnarray*}
M^*_1 &=&  \max\{|b^*(i)-b^*(j)|\},\\
M^*_2 &=& \max\{-2b^*(i) + |b^*(i) - b^*(j)| +|d^*(i)-d^*(j)|\}.
\end{eqnarray*}

Suppose $p - (4a^* + M^*_1)q > 0$, or $p - (4a^* - M^*_1)q < - M^*_2$ and $p - 4a^*q <0$. Then for $n \gg 0$ we have
$$d_-[J_{\K}(n)]=A^*n^2+B^*(n)n+D^*(n)$$
where $A^*$ is a constant, and $B^*(n), D^*(n)$ are periodic functions, and $B^*(n) \ge 0$. Moreover, if $js^*_K \subset bs_K$ then $js^*_{\K} \subset bs_{\K}$.
}

\bibliographystyle{hamsplain} \bibliography{biblio}
\end{document}